\documentclass[11pt]{amsart}

\usepackage{amsfonts}
\usepackage{times}
\usepackage{graphicx}
\newtheorem{theorem}{\textbf{{\normalfont THEOREM}}}

\theoremstyle{plain}

\newtheorem{example}{\textbf{{\normalfont EXAMPLE}}}

\newtheorem{proposition}{\textbf{{\normalfont PROPOSITION}}}

\numberwithin{equation}{section}

\begin{document}
\title[Focal Surfaces of A Tubular Surface in $\mathbb{E}^{3}$ ]{\textbf{{\normalfont Some
	Characterizations of Focal Surfaces of A Tubular Surface in $\mathbb{E}^{3}$}}
}
\author{SEZG\.{I}N B\"{U}Y\"{U}KK\"{U}T\"{U}K, \.{I}L\.{I}M K\.{I}\c{S}\.{I}, G\"{U}NAY \"{O}%
ZT\"{U}RK}
\address[S. B\"{u}y\"{u}kk\"{u}t\"{u}k]{
Kocaeli University, G\"{o}lc\"{u}k Vocational School of Higher Education, Kocaeli, TURKEY}
\email{sezginbuyukkutuk@gmail.com}
\address[\.{I}. Ki\c{s}i]{
Kocaeli University, Art and Science Faculty, Department of Mathematics,
Kocaeli, TURKEY}
\email{ilim.ayvaz@kocaeli.edu.tr}
\address[G. \"{O}zt\"{u}rk]{
\.{I}zmir Demokrasi University, Art and Science Faculty, Department of Mathematics,
\.{I}zmir, TURKEY}
\email{gunay.ozturk@idu.edu.tr}

{\scriptsize
\begin{abstract}
Here, we focus on focal surfaces of a tubular surface in Euclidean $%
3-$space $\mathbb{E}^{3}.$ Firstly, we give the tubular surfaces with
respect to Frenet and Darboux frames. Then, we define focal surfaces
of these tubular surfaces. We get some results for these types of surfaces
to become flat and we show that there is no minimal focal surface of a
tubular surface in $\mathbb{E}^{3}$. We give some examples for these type surfaces. Further, we show that $u$-parameter
curves cannot be asymptotic curves and we obtain some results about $v$%
-parameter curves of the focal surface $M^{\ast }$.
\end{abstract}}

\maketitle

\section{\protect\smallskip \textbf{INTRODUCTION}}

Focal surfaces are known as line congruences. The concept of line
congruences is defined for the first time in visualization in 1991 by Hagen
and Pottman \cite{HPD}.

Let $M:X(u,v)$ be a surface defined as a real-valued function and $N\left(
u,v\right) $ be a unit normal vector on the surface. The line congruence is
defined as%
\begin{equation}
C\left( u,v,z\right) =X\left( u,v\right) +zE\left( u,v\right) ,  \label{1.1}
\end{equation}%
where $E\left( u,v\right) $ is the set of unit vectors. For each $\left(
u,v\right) $, the equation (\ref{1.1}) indicates a line congruence and
called generatrix. Here, the parameter $z$ is a marked distance. In
addition, there exist two special points (real, imaginary or unit) on the
generatrix of $C.$ These points are called as focal points which are the
osculator points with generatrix. Therefore, focal surfaces are defined as a
geometric locus of focal points. In general, there exist two focal surfaces.
If $E\left( u,v\right) =N\left( u,v\right) $, then $C=C_{N}$ is normal
congruence. Thus, the focal surface $C_{N}$ has the following parametric
representation%
\begin{equation*}
X^{\ast }_{i}\left( u,v\right) =X\left( u,v\right) +\kappa _{i}^{-1}\left(
u,v\right) N\left( u,v\right) ,
\end{equation*}%
where $\kappa _{1}$ and $\kappa _{2}$ are the principle curvature functions
of the surface $M:X\left( u,v\right) $ \cite{HH}. The center of curvature of
the normal section curve corresponds to a certain level of the normal vector
at a point $X\left( u_{0},v_{0}\right) $ on $M.$ The extreme values are the
center of the curvature of two principle directions. These points correspond
to the focal points. For this reason, the congruence of lines is considered
as a set of lines which are tangent to two surfaces. Also, these two
surfaces are focal surfaces of congruence of the lines. Thus, focal points
of the normal congruence are the centers of curvature of two principle
directions. Some studies can be found about focal curves and focal surfaces
in Euclidean spaces \cite{HH, O, OA2, OA}.

Canal surface with a significant place in geometry provides benefits in
showing human internal organs, long thin objects, surface modelling, CG/ CAD
and graphics. A canal surface $X\left( u,v\right) $ obtained by spin curve $%
\gamma \left( u\right) $ is the combination of the spheres which are
determined by the radius function $r(u)$ and center $\gamma \left( u\right)$. If the radius function is constant, then this surface is called tubular surface. In \cite{S}, the authors studied tubular surface in Euclidean $3-$%
space. Recently, in \cite{KO, KO2, KOA}, the authors have attended to
tubular surfaces in Euclidean $4-$space $\mathbb{E}^{4}.$ Thanks to
geometric structure of these types of surfaces, they are also used in
reshaping and planning the movement lines of the robots (see, \cite{MPSY}%
).

In differential geometry, frame fields are important tools for analyzing
curves and surfaces. Frenet frame is the most familier frame field but there
is also the other frame field such as Darboux frame.
Besides Frenet frame is constructed on a curve with its velocity and the
acceleration vectors, Darboux frame is constructed on a surface with the
curves' velocity vector and the surface' normal vector. There have been so many studies about Darboux frame such as \cite{DDKO, KA, KT}.

In the current work, we study focal surfaces of a tubular surface which are
constructed by the Frenet frame and Darboux frame. We
calculate the mean and Gaussian curvatures of the focal
surfaces and get the necessary and sufficient conditions of these surfaces
to become flat and minimal.

\section{\protect\smallskip \textbf{BASIC CONCEPTS}}

Let $\gamma =\gamma (s):I\subseteq\mathbb{R}\rightarrow \mathbb{E}^{3}$ be a unit speed
curve in the Euclidean space $%
\mathbb{E}^{3}$. Then the derivatives
of the Frenet frame $\left\{ T,N_{1},N_{2}\right\} $ of $\gamma $
(Frenet-Serret formula);%
\begin{equation*}
\left[
\begin{array}{c}
T^{^{\prime }} \\
N_{1}^{^{\prime }} \\
N_{2}^{^{\prime }}%
\end{array}%
\right] =\left[
\begin{array}{ccc}
0 & \kappa & 0 \\
-\kappa & 0 & \tau \\
0 & -\tau & 0%
\end{array}%
\right] \left[
\begin{array}{c}
T \\
N_{1} \\
N_{2}%
\end{array}%
\right] ,
\end{equation*}%
where $\tau $, $\kappa $ are the torsion and curvature of the curve $\gamma $,
respectively \cite{C}.

Let $\gamma :$ $I\rightarrow M$ be a unit speed curve on the surface $M$.
Then, $\left\{ T,Y=T\times N,N\right\}$ is called as Darboux frame which is defined along
the curve $\gamma$. Here $T$ is the tangent vector of $\gamma$ and $N$ is the
unit normal vector of $M$. Darboux frame formulas are expressed as
\begin{equation*}
\left[
\begin{array}{c}
T~^{\prime } \\
Y^{\prime } \\
N^{\prime }%
\end{array}%
\right] =\left[
\begin{array}{ccc}
0 & k_{g} & k_{n} \\
-k_{g} & 0 & \tau _{g} \\
-k_{n} & -\tau _{g} & 0%
\end{array}%
\right] \left[
\begin{array}{c}
T \\
Y \\
N%
\end{array}%
\right] ,
\end{equation*}%
where $k_{g}$ is the geodesic curvature, $k_{n}$ is the normal curvature and
$\tau _{g}$ is the geodesic torsion of the curve $\gamma $.

The relation between the Darboux frame and the Frenet frame is given as
follows:%
\begin{equation*}
\left[
\begin{array}{c}
T \\
Y \\
N%
\end{array}%
\right] =\left[
\begin{array}{ccc}
1 & 0 & 0 \\
0 & \cos \theta  & \sin \theta  \\
0 & -\sin \theta  & \cos \theta
\end{array}%
\right] \left[
\begin{array}{c}
T \\
N_{1} \\
N_{2}%
\end{array}%
\right] ,
\end{equation*}%
where $\theta $ is the angle between the vectors $Y$ and $N_{1}$. Here
Darboux curvatures are defined by $k_{g}=\kappa \cos \theta ,$ $k_{n}=\kappa
\sin \theta $ and $\tau _{g}=\tau -\theta ^{\prime }$.

Let $M$:$X(u,v)$ be a regular surface in $\mathbb{E}^{3}$. The tangent space of $%
M$ is spanned by the vectors $X_{u}$ and $%
X_{v}$ at a point $p=X(u,v)$. The coefficients of the first fundamental form of $M$ are defined as
\begin{equation}
E=\langle X_{u},X_{u}\rangle ,\text{ \ \ }F=\left\langle
X_{u},X_{v}\right\rangle ,\text{ \ \ }G=\left\langle
X_{v},X_{v}\right\rangle ,  \label{A1}
\end{equation}%
where $\left\langle ,\right\rangle $ is the Euclidean inner product. We
set $W^{2}=EG-F^{2}\neq 0$.

Then the unit normal vector field of $M$ is defined as%
\begin{equation}
N=\frac{X_{u}\times X_{v}}{\left\Vert X_{u}\times X_{v}\right\Vert }.
\label{A2}
\end{equation}%
The coeffiicients of the second fundamental form are as follow:
\begin{equation}
l=\left\langle X_{uu},N\right\rangle ,\text{ \ \ }m=\left\langle
X_{uv},N\right\rangle ,\text{ \ \ }n=\left\langle X_{vv},N\right\rangle,
\label{A3}
\end{equation}%
\cite{C}.

The shape operator matrix of a surface is defined as

\begin{equation*}
A_{N}=\left[
\begin{array}{cc}
\frac{l}{E} & \frac{1}{W}\left( m-\frac{F}{E}l\right) \\
\frac{1}{W}\left( m-\frac{F}{E}l\right) & \frac{1}{W^{2}}\left( En-2Fm+\frac{%
	F^{2}}{E}l\right)%
\end{array}%
\right] ,
\end{equation*}%
\cite{Be}.

The Gaussian and the mean curvatures of $M$ are given as

\begin{equation}
K=\frac{ln-m^{2}}{EG-F^{2}},  \label{ilim18}
\end{equation}%
and%
\begin{equation}
H=\frac{En+Gl-2Fm}{2\left( EG-F^{2}\right) },  \label{ilim19}
\end{equation}%
respectively \cite{C, IAO}.

Let $\gamma :$ $I\rightarrow M$ be a unit speed curve on the surface $M$.
Then $\gamma $ is an asymptotic curve for which the normal curvature
vanishes in the direction $\gamma ^{\prime }$. Recall that $\gamma $ is
asymptotic if and only if $\gamma ^{\prime \prime }$ is perpendicular to the normal
vector $N$ of the surface $M$. Furthermore, $\gamma $ is a geodesic curve on
$M$ if the tangential component $(\gamma ^{\prime \prime })^{T}$ of the
acceleration of $\gamma $ vanishes \cite{C}.

\section{\protect\smallskip \textbf{FOCAL SURFACE OF A TUBULAR SURFACE WITH FRENET FRAME}}

Tubular surface with Frenet frame was studied by \"{O}zt\"{u}rk et al.,\ in
\cite{OBBA}. In this section, we handle the tubular surface with Frenet
frame and give the focal surface of this surface in $\mathbb{E}^{3}$.

Let $\gamma \left( u\right) =\left( \gamma _{1}\left( u\right) ,\gamma
_{2}\left( u\right) ,0\right) \subset \mathbb{E}^{3}$ be a curve
parametrized by arclength. Then the famous Frenet formulas become%
\begin{equation*}
\begin{array}{c}
\gamma ^{\prime }=T, \\
T^{\prime }=\kappa N_{1}, \\
N_{1}^{\prime }=-\kappa T, \\
N_{2}^{\prime }=0.%
\end{array}%
\end{equation*}

The tubular surface with respect to Frenet frame has the parametrization:
\begin{equation}
M:X\left( u,v\right) =\gamma \left( u\right) +r\left( \cos vN_{1}\left(
u\right) +\sin vN_{2}\right) ,  \label{e1}
\end{equation}%
where $r=const.$ is the radius of the spheres. The tangent space of $M$ at
a point $p=X(u,v)$ is spanned by
\begin{eqnarray}
X_{u} &=&\left( 1-\kappa (u)r\cos v\right) T,  \label{e3} \\
X_{v} &=&-r\sin vN_{1}+r\cos vN_{2}.  \notag
\end{eqnarray}%
Then coefficients of the first fundamental form are
\begin{equation}
E=\left( 1-\kappa (u)r\cos v\right) ^{2},\text{ \ \ }F=0,\text{ \ \ }G=r^{2},
\label{e10}
\end{equation}%
where $W^{2}=EG-F^{2}=\left( 1-\kappa (u)r\cos v\right) ^{2}r^{2}$ \cite%
{OBBA}.

\begin{proposition}
	\cite{OBBA} $X(u,v)$ is a regular tubular surface patch if and only if \ $%
	1-\kappa (u)r\cos v\neq 0$.
\end{proposition}

The unit normal vector field and the second partial derivatives of $M$ are
obtained as%
\begin{equation*}
N=-\cos vN_{1}-\sin vN_{2}
\end{equation*}%
and%
\begin{eqnarray}
X_{uu} &=&-\kappa ^{\prime }(u)r\cos vT+\kappa (u)\left( 1-\kappa (u)r\cos
v\right) N_{1},  \notag \\
X_{uv} &=&\kappa (u)r\sin vT,  \label{e8} \\
X_{vv} &=&-r\cos vN_{1}-r\sin vN_{2},  \notag
\end{eqnarray}%
respectively.

Then the coefficients of the second fundamental form become%
\begin{equation}
l=-\kappa (u)\left( 1-\kappa (u)r\cos v\right) \cos v,\text{ \ \ }m=0,\text{ \
	\ }n=r.  \label{e9}
\end{equation}%
Thus, from the equations (\ref{e10}) and (\ref{e9}),\ the Gaussian and mean
curvature functions of $M$ are calculated as%
\begin{equation*}
K=\frac{-\kappa (u)\cos v}{r\left( 1-\kappa (u)r\cos v\right) },
\end{equation*}%
and%
\begin{equation*}
H=\frac{1-2\kappa (u)r\cos v}{2r\left( 1-\kappa (u)r\cos v\right) },
\end{equation*}%
respectively.

The shape operator matrix of the surface $M$ is as follows:%
\begin{equation}
A_{N}=\left[
\begin{array}{cc}
\frac{-\kappa (u)\cos v}{1-\kappa (u)r\cos v} & 0 \\
0 & \frac{1}{r}%
\end{array}%
\right]  \label{z4}
\end{equation}%
\cite{OBBA}.

From now on, we can give the parametrization of the focal surface $M^{\ast }$
of $M$ by obtaining the principal curvature functions of $M$.

Using (\ref{z4}), we get the principal curvature functions as%
\begin{equation}
\kappa _{1}=\frac{1}{r},\ \kappa _{2}=\frac{-\kappa (u)\cos v}{1-\kappa
	(u)r\cos v}.  \label{z5}
\end{equation}%
From the definition of the focal surface of a given surface and using the
equation (\ref{z5}), we obtain the focal surface $M^{\ast }$ of $M$ as%
\begin{equation}
X^{\ast }\left( u,v\right) =\gamma \left( u\right) +\frac{1}{\kappa (u)\cos v%
}\left( \cos vN_{1}\left( u\right) +\sin vN_{2}\right) .  \label{z3}
\end{equation}%
The tangent space of the focal surface $M^{\ast }$ is spanned by
\begin{equation}
\left( X^{\ast }\right) _{u}=\frac{-\kappa ^{\prime }(u)}{\kappa ^{2}(u)}%
\left( N_{1}+\tan vN_{2}\right) ,  \label{e14}
\end{equation}%
and%
\begin{equation*}
\left( X^{\ast }\right) _{v}=\frac{1}{\kappa (u)\cos ^{2}v}N_{2}.
\end{equation*}%
Thus from (\ref{e14}), the coefficients of the first fundamental form are
obtained as
\begin{equation}
E^{\ast }=\frac{\left( \kappa ^{\prime }(u)\right) ^{2}}{\kappa ^{4}(u)\cos
	^{2}v},\text{ \ \ }F^{\ast }=\frac{-\kappa ^{\prime }(u)\sin v}{\kappa
	^{3}(u)\cos ^{3}v},\text{ \ \ }G^{\ast }=\frac{1}{\kappa ^{2}(u)\cos ^{4}v},
\label{16}
\end{equation}%
where $\left( W^{\ast }\right) ^{2}=\frac{\left( \kappa ^{\prime }(u)\right)
	^{2}}{\kappa ^{6}(u)\cos ^{4}v}\neq 0$. Further, from (\ref{A2}) we find the
normal vector of $M^{\ast }$
\begin{equation}
N^{\ast }=-T.  \label{e17}
\end{equation}
The second partial derivatives of $X^{\ast }\left( u,v\right) $ are as in
the following:%
\begin{eqnarray}
\left( X^{\ast }\right) _{uu} &=&\frac{\kappa ^{\prime }(u)}{\kappa (u)}T+%
\frac{-\kappa ^{\prime \prime }(u)\kappa (u)+2\left( \kappa ^{\prime
	}(u)\right) ^{2}}{\kappa ^{3}(u)}N_{1}  \notag \\
&&\text{ \ \ \ \ \ \ \ \ \ \ \ }+\frac{-\kappa ^{\prime \prime }(u)\kappa
	(u)+2\left( \kappa ^{\prime }(u)\right) ^{2}}{\kappa ^{3}(u)}\tan vN_{2}
\notag \\
\left( X^{\ast }\right) _{uv} &=&-\frac{\kappa ^{\prime }(u)}{\kappa ^{2}(u)}%
\left( 1+\tan ^{2}v\right) N_{2},  \label{e18} \\
\left( X^{\ast }\right) _{vv} &=&\frac{2\sin v}{\kappa (u)\cos ^{3}v}N_{2}.
\notag
\end{eqnarray}%
Hence, from (\ref{e17}) and (\ref{e18}), we get the coefficients of the
second fundamental form of the focal surface $M^{\ast }$ as%
\begin{equation}
l^{\ast }=-\frac{\kappa ^{\prime }(u)}{\kappa (u)},\text{ \ \ }m^{\ast }=0,%
\text{ \ \ }n^{\ast }=0.  \label{e19}
\end{equation}

Thus, we can give the following results:

\begin{theorem}
	Let $M$ be a tubular surface with Frenet frame given with the
	parametrization (\ref{e1}) and $M^{\ast }$ be the focal surface of $M$ with
	the parametrization (\ref{z3}) in $\mathbb{E}^{3}$. Then the Gaussian
	curvature of $M^{\ast }$ vanishes, so the focal surface is flat.
\end{theorem}

\begin{proof}
	Let $M^{\ast }$ be the focal surface of $M$ with the parametrization (\ref%
	{z3}) in $\mathbb{E}^{3}$. Using the equations (\ref{ilim18}), (\ref{16})
	and (\ref{e19}), we get $K^{\ast }=0$, which completes the proof.
\end{proof}

\begin{theorem}
	Let $M$ be a tubular surface with Frenet frame given with the
	parametrization (\ref{e1}) and $M^{\ast }$ be the focal surface of $M$ with
	the parametrization (\ref{z3}) in $\mathbb{E}^{3}$. Then the mean curvature
	of $M^{\ast }$ is
	\begin{equation}
	H^{\ast }=-\frac{\kappa ^{3}(u)}{2\kappa ^{\prime }(u)}.  \label{e20}
	\end{equation}
\end{theorem}

\begin{proof}
	Let $M^{\ast }$ be the focal surface of $M$ with the parametrization (\ref%
	{z3}) in $\mathbb{E}^{3}$. Using the equations (\ref{ilim19}), (\ref{16})
	and (\ref{e19}), we get the result.
\end{proof}

\begin{theorem}
	There is no minimal focal surface $M^{\ast }$ of the tubular surface$\ M$.
\end{theorem}

\begin{proof}
	Assume that the focal surface $M^{\ast }$ of the surface $M$ is minimal.
	From the equation (\ref{e20}), the curvature function $\kappa $ according to
	Frenet frame vanishes identically, this is a contradiction. Thus, there is
	no minimal focal surface of the tubular surface $M$.
\end{proof}

\begin{theorem}
	Let $M$ be a tubular surface with Frenet frame given with the
	parametrization (\ref{e1}) and $M^{\ast }$ be the focal surface of $M$ with
	the parametrization (\ref{z3}) in $\mathbb{E}^{3}$. Then the focal surface $%
	M^{\ast }$ has constant mean curvature if and only if the curvature function
	$\kappa (u)$ satisfies
	\begin{equation}
	\kappa (u)=\pm \frac{\sqrt{\left( u+c_{1}c\right) c}}{u+c_{1}c},  \label{e21}
	\end{equation}
	where $c,c_{1}$ are real constants.
\end{theorem}

\begin{proof}
	Let $M^{\ast }$ be the focal surface of $M$ with the parametrization (\ref%
	{z3}) in $\mathbb{E}^{3}$. From the equation (\ref{e20}), we get the
	differential equation%
	\begin{equation*}
	\kappa ^{3}(u)+2\kappa ^{\prime }(u)c=0,
	\end{equation*}
	which has a non-trivial solution (\ref{e21}).
\end{proof}

\begin{example}
	Let us consider the unit speed planar curve with the parameterization
	\begin{equation*}
	\gamma (u)=\left( \left( \frac{u}{\sqrt{2}}+1\right) \cos \left( \ln \left(
	\frac{u}{\sqrt{2}}+1\right) \right) ,\left( \frac{u}{\sqrt{2}}+1\right) \sin
	\left( \ln \left( \frac{u}{\sqrt{2}}+1\right) \right) ,0\right) .
	\end{equation*}%
	The Frenet apparatus of this curve are determined by
	\begin{equation*}
	\begin{array}{c}
	T(u)=\gamma ^{\prime }(u)=\frac{1}{\sqrt{2}}\left( \cos \left( \ln \left(
	\frac{u}{\sqrt{2}}+1\right) \right) -\sin \left( \ln \left( \frac{u}{\sqrt{2}%
	}+1\right) \right) ,\right. \\
	\text{ \ \ \ \ \ \ \ \ \ \ \ \ \ \ \ \ \ \ \ \ \ \ \ \ \ } \text{\ \ }\left.
	\sin \left( \ln \left( \frac{u}{\sqrt{2}}+1\right) \right) +\cos \left( \ln
	\left( \frac{u}{\sqrt{2}}+1\right) \right) ,0\right) ,%
	\end{array}%
	\end{equation*}%
	\begin{equation*}
	\begin{array}{c}
	N_{1}(u)=\frac{1}{\sqrt{2}}\left( \sin \left( \ln \left( \frac{u}{\sqrt{2}}%
	+1\right) \right) -\cos \left( \ln \left( \frac{u}{\sqrt{2}}+1\right)
	\right) ,\right. \\
	\text{ \ \ \ }\ \text{\ \ \ \ \ \ \ \ \ \ \ \ \ }\ \text{\ }\left. \cos
	\left( \ln \left( \frac{u}{\sqrt{2}}+1\right) \right) -\sin \left( \ln
	\left( \frac{u}{\sqrt{2}}+1\right) \right) ,0\right) ,%
	\end{array}%
	\end{equation*}%
	\begin{equation*}
	N_{2}(u)=\left( 0,0,1\right) ,
	\end{equation*}%
	\begin{equation*}
	\kappa (u)=\frac{1}{u+\sqrt{2}},\text{ \ \ }\tau (u)=0.
	\end{equation*}%
Hence, the parameterization of tubular surface around the curve $\gamma (u)$ can be written with the Frenet frame as
{\scriptsize \begin{equation*}
		\begin{array}{c}
		X\left( u,v\right) =\left( \left( \frac{u}{\sqrt{2}}+1\right) \cos \left(
		\ln \left( \frac{u}{\sqrt{2}}+1\right) \right) -\frac{r}{\sqrt{2}}\cos
		v\left( \sin \left( \ln \left( \frac{u}{\sqrt{2}}+1\right) \right) +\cos
		\left( \ln \left( \frac{u}{\sqrt{2}}+1\right) \right) \right) ,\right. \\
		\text{ \ \ \ \ \ \ \ \ \ \ \ \ \ \ \ }\left( \frac{u}{\sqrt{2}}+1\right)
		\sin \left( \ln \left( \frac{u}{\sqrt{2}}+1\right) \right) +\frac{r}{\sqrt{2}%
		}\cos v\left( \cos \left( \ln \left( \frac{u}{\sqrt{2}}+1\right) \right)
		-\sin \left( \ln \left( \frac{u}{\sqrt{2}}+1\right) \right) \right) , \\
		\text{ \ \ \ \ \ \ \ }\left. r\sin v\right) .%
		\end{array}%
		\end{equation*}}
Also, the parameterization of the focal surface of this surface is obtained as
\begin{equation*}
\begin{array}{c}
	X^{\ast }(u,v)=\left( \left( \frac{u^{2}+2\sqrt{2}u+1}{\sqrt{2}\left( u+%
		\sqrt{2}\right) }\right) \cos \left( \ln \left( \frac{u}{\sqrt{2}}+1\right)
	\right) -\frac{1}{\sqrt{2}\left( u+\sqrt{2}\right) }\sin \left( \ln \left(
	\frac{u}{\sqrt{2}}+1\right) \right) ,\right. \\
	\text{ \ \ \ \ \ \ \ \ \ \ \ \ \ \ \ }\left( \frac{u^{2}+2\sqrt{2}u+1}{\sqrt{%
			2}\left( u+\sqrt{2}\right) }\right) \sin \left( \ln \left( \frac{u}{\sqrt{2}}%
	+1\right) \right) +\frac{1}{\sqrt{2}\left( u+\sqrt{2}\right) }\cos \left(
	\ln \left( \frac{u}{\sqrt{2}}+1\right) \right) , \\
	\left. \frac{1}{u+\sqrt{2}}\tan v\right)%
\end{array}.
\end{equation*}%
	Further, by taking radius $r=\sqrt{2},$ we plot the tubular surface and its
	focal surface in $\mathbb{E}^{3}$:
	\begin{figure}[tbph]
		\centering
		\includegraphics[width =10cm, angle=0]{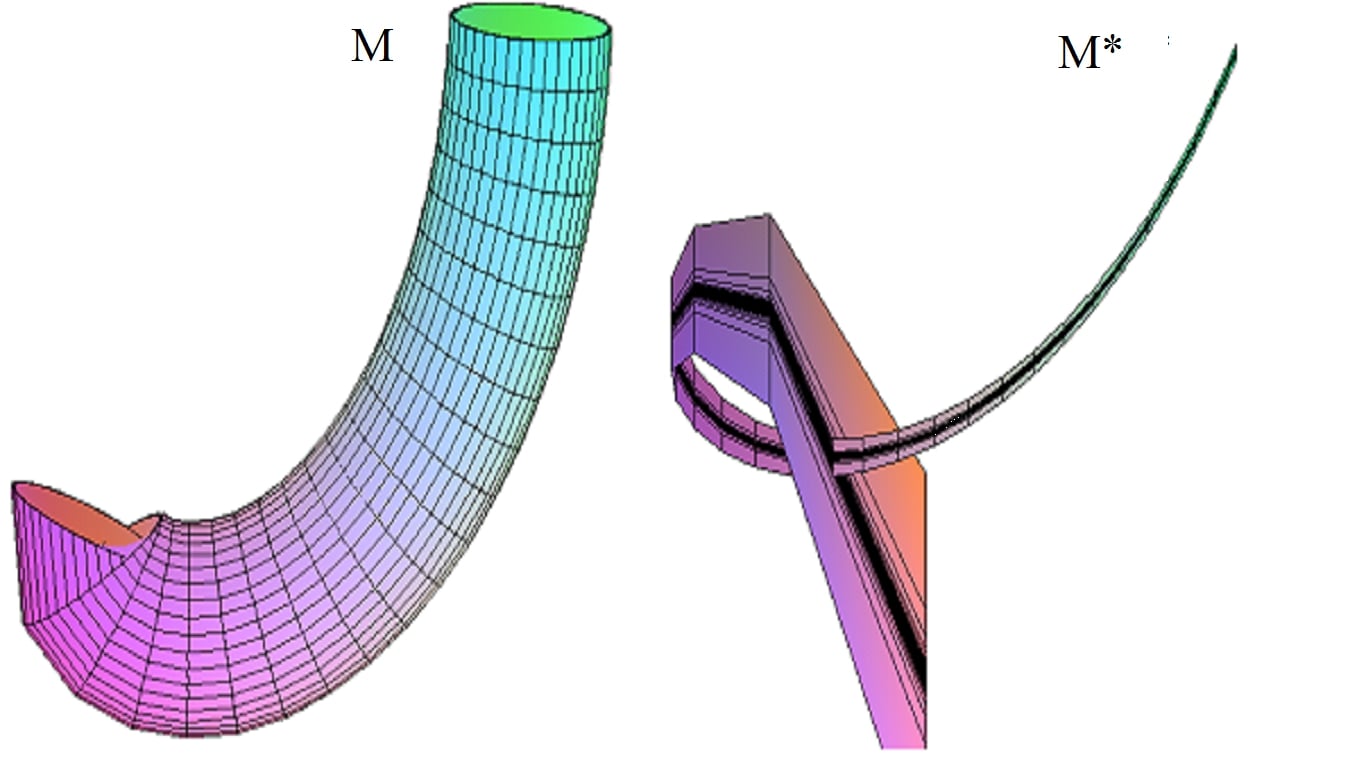}
		\caption{Tubular surface $M$ and the focal surface $M^{\ast }$}
	\end{figure}
\end{example}

Now, we obtain the following results for parameter curves on the
focal surface $M^{\ast }$.

\begin{theorem}
	Let $M$ be a tubular surface with Frenet frame given with the
	parametrization (\ref{e1}) and $M^{\ast }$ be the focal surface of $M$ with
	the parametrization (\ref{z3}) in $\mathbb{E}^{3}$. Then,
	
	i) $u$-parameter curves of the focal surface $M^{\ast }$ cannot be
	asymptotic curves.
	
	ii) $v$-parameter curves of the focal surface $M^{\ast }$ are asymptotic
	curves.
\end{theorem}

\begin{proof}
	Let $M^{\ast }$ be the focal surface of $M$ with the parametrization (\ref%
	{z3}) in $\mathbb{E}^{3}$.
	
	i) From the definition of an asymptotic curve, using (\ref{e17}) and (\ref%
	{e18}), we get $\left\langle \left( X^{\ast }\right) _{uu},N^{\ast
	}\right\rangle =\frac{\kappa ^{\prime }(u)}{\kappa (u)}=0$ if and only if $%
	\kappa $ is constant which contradicts the regularity of $M^{\ast }$. Thus, $%
	u$-parameter curves of the focal surface $M^{\ast }$ cannot be asymptotic
	curves.
	
	ii) Again using the same equations, we obtain $\left\langle \left( X^{\ast
	}\right) _{vv},N^{\ast }\right\rangle =0$ which means $v$-parameter curves
	of the focal surface $M^{\ast }$ are asymptotic curves.
\end{proof}

\begin{theorem}
	Let $M$ be a tubular surface with Frenet frame given with the
	parametrization (\ref{e1}) and $M^{\ast }$ be the focal surface of $M$ with
	the parametrization (\ref{z3}) in $\mathbb{E}^{3}$. Then,
	
	i) $u$-parameter curves of the focal surface $M^{\ast }$ are geodesic curves
	if and only if the equation%
	\begin{equation}
	\kappa (u)=-\frac{1}{c_{1}u+c_{2}},  \label{ozgu1}
	\end{equation}%
	held for the curvature of $\gamma $, where $c_{1},c_{2}$ are real constants.
	
	ii) $v$-parameter curves of the focal surface $M^{\ast }$ are geodesic
	curves if and only if \ $v=t\pi ,$ $t\in \mathbb{Z}$.
\end{theorem}

\begin{proof}
	i) From (\ref{e17}) and (\ref{e18}), we get $\left( X^{\ast }\right)
	_{uu}\wedge N^{\ast }=0$ if and only if $-\kappa ^{\prime \prime }\kappa
	+2(\kappa ^{\prime })^{2}=0$, which has a solution (\ref{ozgu1}) for real
	constants $c_{1}$ and $c_{2}$.
	
	ii) Again using the same equations, we obtain $\left( X^{\ast }\right)
	_{vv}\wedge N^{\ast }=0$ if and only if $\sin v=0$ which completes the proof.
\end{proof}

\section{\protect\smallskip \textbf{FOCAL SURFACE OF A TUBULAR SURFACE WITH DARBOUX FRAME}}

Tubular surface with the Darboux frame was studied by Do\u{g}an and Yayl\i\
in \cite{DY2}. In this part, we handle the tubular surface according to
the Darboux frame and give the focal surface of this surface in $\mathbb{E}%
^{3}$.

Let $\gamma \left( u\right) =\left( \gamma _{1}\left( u\right) ,\gamma
_{2}\left( u\right) ,\gamma _{3}\left( u\right) \right) $ be a unit speed
curve on a surface $S.$ The tubular surface with respect to Darboux frame has
the following parametrization:%
\begin{equation}
M:X\left( u,v\right) =\gamma \left( u\right) +r\left( \cos vY\left( u\right)
+\sin vU\left( u\right) \right) ,  \label{d1}
\end{equation}%
where $r=const.$ is the radius of the spheres and $U$ is the unit normal of
the surface $S$ along the curve $\gamma $. Then the tangent space of $M$ is spanned by $X_{u}$ and $X_{v}$ at a point $p=X(u,v)$:%
\begin{eqnarray}
X_{u} &=&\left( 1-br\right) T-r\tau _{g}\sin vY+r\tau _{g}\cos vU,
\label{d2} \\
X_{v} &=&-r\sin vY+r\cos vU,  \notag
\end{eqnarray}%
where%
\begin{equation}
b\left( u,v\right) =k_{g}\left( u\right) \cos v+k_{n}(u)\sin v.  \label{d2a}
\end{equation}%
Then coefficients of the first fundamental form become%
\begin{equation}
E=\left( 1-br\right) ^{2}+r^{2}\tau _{g}^{2},\ \ F=r^{2}\tau _{g},\ \
G=r^{2},  \label{d3}
\end{equation}%
where $W^{2}=EG-F^{2}=\left( 1-br\right) ^{2}r^{2}$ \cite{DY2}.

\begin{proposition}
	\cite{DY2} $X(u,v)$ is a regular tubular surface patch if and only if \ $%
	b\neq \frac{1}{r}$.
\end{proposition}

The unit normal vector field and the second partial derivatives of $M$ are
obtained as
\begin{equation}
N=-\cos vY-\sin vU.  \label{d4}
\end{equation}%
and%
\begin{eqnarray}
X_{uu} &=&\left( -b_{u}r-r\tau _{g}b_{v}\right) T+\left( k_{g}\left(
1-br\right) -r\tau _{g}^{\prime }\sin v-r\tau _{g}^{2}\cos v\right) Y  \notag
\\
&&\text{ \ \ \ \ \ \ \ \ \ \ \ \ \ \ \ \ \ \ \ \ \ \ \ \ \ \ }+\left(
k_{n}\left( 1-br\right) +r\tau _{g}^{\prime }\cos v-r\tau _{g}^{2}\sin
v\right) U,  \notag \\
X_{uv} &=&-b_{v}rT-r\tau _{g}\cos vY-r\tau _{g}\sin vU,  \label{d5} \\
X_{vv} &=&-r\cos vY-r\sin vU,  \notag
\end{eqnarray}%
respectively.

Then the coefficients of the second fundamental form become%
\begin{equation}
l=-\left( 1-br\right) b+r\tau _{g}^{2},\text{ \ \ }m=r\tau _{g},\text{ \ \ }%
n=r.  \label{d6}
\end{equation}

\bigskip

Thus, from the equations (\ref{d3}) and (\ref{d6}),\ the Gaussian and mean
curvature functions of $M$ are calculated as%
\begin{equation}
K=\frac{b}{r\left( br-1\right) },  \label{d7}
\end{equation}%
and%
\begin{equation}
H=\frac{1-2br}{2\left( 1-br\right) r}  \label{d8}
\end{equation}%
respectively \cite{DY2}.

Now, we focus on the parametrization of the focal surface $M^{\ast }$ of $M$
by obtaining the principal curvature functions of $M$.

Using (\ref{d7}), (\ref{d8}), and the equation $\kappa _{i}=H\pm \sqrt{%
	H^{2}-K}$, $(i=1,2)$, we get the principal curvature functions as
\begin{equation}
\kappa _{1}=\frac{1}{r},\text{ \ }\kappa _{2}=\frac{b}{br-1}.  \label{d9}
\end{equation}%
From the definition of the focal surface of a given surface and using the
equation (\ref{d9}), we obtain the focal surface $M^{\ast }$ of $M$ as
\begin{equation}
X^{\ast }\left( u,v\right) =\gamma \left( u\right) +\frac{1}{b(u,v)}\left(
\cos vY(u)+\sin vU(u)\right) ,  \label{d10}
\end{equation}%
where $b\left( u,v\right) =k_{g}\left( u\right) \cos v+k_{n}(u)\sin v.$

The tangent space of the focal surface $M^{\ast }$ is spanned by the vectors
\begin{equation}
\left( X^{\ast }\right) _{u}=\left( -\frac{b_{u}}{b^{2}}\cos v-\frac{1}{b}%
\tau _{g}\sin v\right) Y+\left( -\frac{b_{u}}{b^{2}}\sin v+\frac{1}{b}\tau
_{g}\cos v\right) U,  \label{d11}
\end{equation}%
and%
\begin{equation}
\left( X^{\ast }\right) _{v}=\left( -\frac{b_{v}}{b^{2}}\cos v-\frac{1}{b}%
\sin v\right) Y+\left( -\frac{b_{v}}{b^{2}}\sin v+\frac{1}{b}\cos v\right) U.
\label{d12}
\end{equation}%
Thus from (\ref{d11}) and (\ref{d12}), the coefficients of the first
fundamental form are obtained as follows:%
\begin{equation}
E^{\ast }=\frac{b_{u}^{2}+b^{2}\tau _{g}^{2}}{b^{4}},\text{ \ \ }F^{\ast }=%
\frac{b_{u}b_{v}+b^{2}\tau _{g}}{b^{4}},\text{ \ \ }G^{\ast }=\frac{%
	b_{v}^{2}+b^{2}}{b^{4}},  \label{d13}
\end{equation}%
where $\left( W^{\ast }\right) ^{2}=\frac{1}{b^{6}}\left( b_{u}-b_{v}\tau
_{g}\right) ^{2}$.

\begin{proposition}
	$X^{\ast }(u,v)$ is a regular tubular surface patch if and only if \ $\left(
	k_{g}^{\prime }-k_{n}\tau _{g}\right) \cos v+\left( k_{n}^{\prime
	}+k_{g}\tau _{g}\right) \sin v\neq 0$.
\end{proposition}

\begin{proof}
	Substituting the partial derivatives of the equation (\ref{d2a}) in $W^{\ast
	}$, we get the result.
\end{proof}

One can find the normal vector of the focal surface $M^{\ast }$ as
\begin{equation}
N^{\ast }=T.  \label{d14}
\end{equation}%
The second partial derivatives of $X^{\ast }\left( u,v\right) $ are%
{\footnotesize \begin{eqnarray}
	\left( X^{\ast }\right) _{uu} &=&\left( \frac{b_{u}-b_{v}\tau _{g}}{b}%
	\right) T  \notag \\
	&&+\left( \frac{-b_{uu}b^{2}\cos v+2bb_{u}^{2}\cos v+2b_{u}b^{2}\tau
		_{g}\sin v-b^{3}\tau _{g}^{2}\cos v-b^{3}\tau _{g}^{\prime }\sin v}{b^{4}}%
	\right) Y \notag \\
	&&+\left( \frac{-b_{uu}b^{2}\sin v+2bb_{u}^{2}\sin v-2b_{u}b^{2}\tau
		_{g}\cos v-b^{3}\tau _{g}^{2}\sin v+b^{3}\tau _{g}^{\prime }\cos v}{b^{4}}%
	\right) U,  \label{m1} \\
	\left( X^{\ast }\right) _{uv} &=&\left( \frac{-b_{uv}b^{2}\cos
		v+2bb_{u}b_{v}\cos v+b_{u}b^{2}\sin v+b_{v}b^{2}\tau _{g}\sin v-b^{3}\tau
		_{g}\cos v}{b^{4}}\right) Y  \notag \\
	&&+\left( \frac{-b_{uv}b^{2}\sin v+2bb_{u}b_{v}\sin v-b_{u}b^{2}\cos
		v-b_{v}b^{2}\tau _{g}\cos v-b^{3}\tau _{g}\sin v}{b^{4}}\right) U,  \notag
	\\
	\left( X^{\ast }\right) _{vv} &=&\left( \frac{-b_{vv}b^{2}\cos
		v+2bb_{v}^{2}\cos v+2b_{v}b^{2}\sin v-b^{3}\cos v}{b^{4}}\right) Y  \notag \\
	&&+\left( \frac{-b_{vv}b^{2}\sin v+2bb_{v}^{2}\sin v-2b_{v}b^{2}\cos
		v-b^{3}\sin v}{b^{4}}\right) U.  \notag
	\end{eqnarray}}

Hence, from (\ref{d13}) and (\ref{m1}), we get the coefficients of the
second fundamental form of the focal surface $M^{\ast }$ as
\begin{equation}
l^{\ast }=\frac{b_{u}-b_{v}\tau _{g}}{b},\text{ \ \ }m^{\ast }=0,\text{ \ \ }%
n^{\ast }=0.  \label{d15}
\end{equation}%
Thus, we can give the following results:

\begin{theorem}
	Let $M$ be a tubular surface according to Darboux frame given with the
	parametrization (\ref{d1}) and $M^{\ast }$ be the focal surface of $M$ with
	the parametrization (\ref{d10}) in $\mathbb{E}^{3}$. Then the Gaussian
	curvature of $M^{\ast }$ vanishes, so the focal surface is flat.
\end{theorem}

\begin{proof}
	Let $M^{\ast }$ be the focal surface of $M$ with the parametrization (\ref%
	{d10}) in $\mathbb{E}^{3}$. Using the equations (\ref{d13}), (\ref{d15}) and
	(\ref{ilim18}), we get $K^{\ast }=0$, which completes the proof.
\end{proof}

\begin{theorem}
	Let $M$ be a tubular surface according to Darboux frame given with the
	parametrization (\ref{d1}) and $M^{\ast }$ be the focal surface of $M$ with
	the parametrization (\ref{d10}) in $\mathbb{E}^{3}$. Then the mean curvature
	of $M^{\ast }$ is
	\begin{equation}
	H^{\ast }=\frac{\left( b_{v}^{2}+b^{2}\right) b}{2\left( b_{u}-b_{v}\tau
		_{g}\right) }.  \label{d16}
	\end{equation}
\end{theorem}

\begin{proof}
	Let $M^{\ast }$ be the focal surface of $M$ with the parametrization (\ref%
	{d10}) in $\mathbb{E}^{3}$. Using the equations (\ref{ilim19}), (\ref{d13})
	and (\ref{d15}), we get the result.
\end{proof}

\begin{theorem}
	There is no minimal focal surface of the tubular surface$\ M$.
\end{theorem}

\begin{proof}
	Assume that the focal surface $M^{\ast }$ of the surface $M$ is minimal.
	From the equation (\ref{d16}), the curvature functions $H^{\ast
	}=0\Leftrightarrow b=0$, a contradiction. Thus, there is no minimal focal
	surface of the tubular surface $M$.
\end{proof}

\begin{example}
	Let us consider the unit speed curve with the parameterization
	\begin{equation*}
	\gamma (u)=\left( \cos \frac{u}{\sqrt{2}},\sin \frac{u}{\sqrt{2}},\frac{u}{%
		\sqrt{2}}\right) .
	\end{equation*}%
	The Darboux frame and curvatures $k_{g}$, $k_{n}$ of this curve are
	determined by
	\begin{equation*}
	\begin{array}{c}
	T(u)=\gamma ^{\prime }(u)=\frac{1}{\sqrt{2}}\left( -\sin \frac{u}{\sqrt{2}}%
	,\cos \frac{u}{\sqrt{2}},1\right),
	\end{array}%
	\end{equation*}%
	\begin{equation*}
	\begin{array}{c}
	Y(u)=\frac{1}{\sqrt{2}}\left( -\cos \frac{u}{\sqrt{2}}+\frac{1}{\sqrt{2}}%
	\sin \frac{u}{\sqrt{2}},-\sin \frac{u}{\sqrt{2}}-\frac{1}{\sqrt{2}}\cos
	\frac{u}{\sqrt{2}},\frac{1}{\sqrt{2}}\right) ,%
	\end{array}%
	\end{equation*}%
	\begin{equation*}
	U(u)=\frac{1}{\sqrt{2}}\left( \cos \frac{u}{\sqrt{2}}+\frac{1}{\sqrt{2}}\sin
	\frac{u}{\sqrt{2}},\sin \frac{u}{\sqrt{2}}-\frac{1}{\sqrt{2}}\cos \frac{u}{%
		\sqrt{2}},\frac{1}{\sqrt{2}}\right) ,
	\end{equation*}%
	\begin{equation*}
	k_{g}(u)=\frac{1}{2\sqrt{2}},\text{ \ \ }k_{n}(u)=\frac{1}{2\sqrt{2}}.
	\end{equation*}%
	Hence, the parameterization of tubular surface around the curve $\gamma (u)$ can be written according to Darboux frame as
	\begin{equation*}
	\begin{array}{c}
	X\left( u,v\right) =\left( \cos \frac{u}{\sqrt{2}}\left( 1-\cos v+\sin
	v\right) +\frac{1}{\sqrt{2}}\sin \frac{u}{\sqrt{2}}\left( \cos v+\sin
	v\right) ,\right. \\
	\text{ \ \ \ \ \ \ \ \ \ \ \ \ \ \ \ \ \ }\sin \frac{u}{\sqrt{2}}\left(
	1-\cos v+\sin v\right) -\frac{1}{\sqrt{2}}\cos \frac{u}{\sqrt{2}}\left( \cos
	v+\sin v\right) , \\
	\text{ \ \ \ \ \ \ \ }\left. \frac{u+\cos v+\sin v}{\sqrt{2}}\right) .%
	\end{array}%
	\end{equation*}%
	Also, the parameterization of the focal surface of this surface is obtained
	as
	{\small \begin{equation*}
		\begin{array}{c}
		X^{\ast }(u,v)=\left( \cos \frac{u}{\sqrt{2}}\left( \frac{3\sin v-\cos v}{\cos v+\sin v}\right) +\sqrt{2}\sin \frac{u}{\sqrt{2}},\right.
		\sin \frac{u}{\sqrt{2}}\left( \frac{%
			3\sin v-\cos v}{\cos v+\sin v}\right) -\sqrt{2}\cos \frac{u}{\sqrt{2}},
		\left. \frac{u+2}{\sqrt{2}}\right)%
		\end{array}
		.
		\end{equation*}}
	Further, by taking radius $r=\sqrt{2},$ we plot the tubular surface and its
	focal surface in $\mathbb{E}^{3}$:
	\begin{figure}[tbph]
		\centering
		\includegraphics[width =10cm, angle=0]{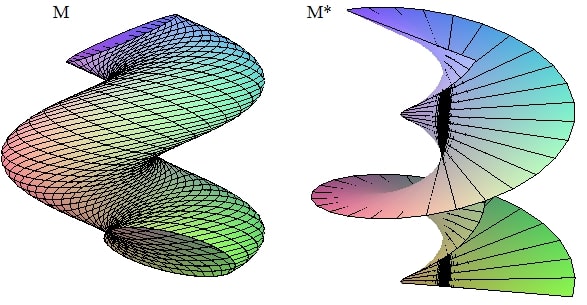}
		\caption{Tubular surface $M$ and the focal surface $M^{\ast }$}
	\end{figure}
\end{example}

\vskip 30mm

\begin{theorem}
	Let $M$ be a tubular surface with Darboux frame given with the
	parametrization (\ref{d1}) and $M^{\ast }$ be the focal surface of $M$ with
	the parametrization (\ref{d10}) in $\mathbb{E}^{3}$. Then,
	
	i) $u$-parameter curves of the focal surface $M^{\ast }$ cannot be
	asymptotic curves.
	
	ii) $v$-parameter curves of the focal surface $M^{\ast }$ are asymptotic
	curves.
\end{theorem}

\begin{proof}
	Let $M^{\ast }$ be the focal surface of $M$ with the parametrization (\ref%
	{d10}) in $\mathbb{E}^{3}$.
	
	i) Using the equations (\ref{d14}), (\ref{m1}) and the definition of
	asymptotic curve, $\left\langle \left( X^{\ast }\right) _{uu},N^{\ast
	}\right\rangle =0$ if and only if $b_{u}-b_{v}\tau _{g}=0.$ But, this
	contradicts the regularity of $M^{\ast }$. Thus, $u$-parameter curves of the
	focal surface $M^{\ast }$ cannot be asymptotic curves.
	
	ii) Again using the same equations, we obtain $\left\langle \left( X^{\ast
	}\right) _{vv},N^{\ast }\right\rangle =0$ which means $v$-parameter curves
	of the focal surface $M^{\ast }$ are asymptotic curves.
\end{proof}

\begin{theorem}
	Let $M$ be a tubular surface with Darboux frame given with the
	parametrization (\ref{d1}) and $M^{\ast }$ be the focal surface of $M$ with
	the parametrization (\ref{d10}) in $\mathbb{E}^{3}$. Then,
	
	i) $u$-parameter curves of the focal surface $M^{\ast }$ are geodesic curves
	if and only if
	\begin{equation}
	-2\tau _{g}\left( b_{u}\tau _{g}^{\prime }+b\tau _{g}^{\prime \prime
	}\right) +4b\left( \tau _{g}^{\prime }\right) ^{2}-4b\tau _{g}^{4}=0
	\label{d17}
	\end{equation}%
	holds.
	
	ii) $v$-parameter curves of the focal surface $M^{\ast }$ cannot be geodesic
	curves.
\end{theorem}

\begin{proof}
	i) From (\ref{d14}) and (\ref{m1}), $\left( X^{\ast }\right) _{uu}\wedge
	N^{\ast }=0$ if and only if%
	\begin{eqnarray}
	b\left( -b_{{u}u}b+2\left( b_{u}\right) ^{2}-b^{2}\left( \tau _{g}\right)
	^{2}\right) &=&0,  \label{d18} \\
	b^{2}\left( 2b_{u}\tau _{g}-b\tau _{g}^{\prime }\right) &=&0.  \label{d19}
	\end{eqnarray}%
	Therefore, by the use of these relations, we obtain the desired result.
	
	ii) Similarly, with the help of (\ref{d14}) and (\ref{m1}), $\left( X^{\ast
	}\right) _{vv}\wedge N^{\ast }=0$ if and only if%
	\begin{eqnarray}
	b\left( b_{vv}b-2\left( b_{v}\right) ^{2}+b^{2}\right) &=&0,  \label{d20} \\
	b_{v} &=&0.  \label{d21}
	\end{eqnarray}
	
	Using these relations, we get $b=0$ which contradicts the representation of
	the focal surface. Hence, we get the result and this completes the proof.
\end{proof}


\begin{thebibliography}{99}
\bibitem{Be} B. Bulca, \textit{A characterization of surfaces in $\mathbb{E}^{4}$}. Ph.D thesis,  Uluda\u{g} University, Bursa, 2012.

\bibitem{C} PM. Do Carmo, \textit{Differential Geometry of Curves and Surfaces.} Englewood Cliffs, NJ, USA: Prentice-Hall, 1976.

\bibitem{DY2} F. Do\u{g}an and Y. Yayl\i, \textit{Tubes with Darboux Frame.} Int. J. Contemp. Math. Sciences. \textbf{7(16)} (2012), 751-758.

\bibitem{DDKO} M. D\"{u}ld\"{u}l, B. Uyar D\"{u}ld\"{u}l, N. Kuruo\u{g}lu and E. \"{O}zdamar, \textit{Extension of the Darboux Frame into Eucliden 4-Space
	and its Invariants.} Turkish Journal of Mathematics. \textbf{41} (2017), 1628-1639.

\bibitem{HH} H. Hagen and S. Hahmann, \textit{Generalized Focal Surfaces.} A New
Method for Surface Interrogation. Proceedings Visualization'92, Boston,
70-76, 1992.

\bibitem{HPD} H. Hagen, H. Pottmann and A. Divivier, \textit{Visualization
	Functions on a Surface.} Journal of Visualization and Animation. \textbf{2} (1991), 52-58.

\bibitem{IAO} E. \.{I}yig\"{u}n, K. Arslan and G. \"{O}zt\"{u}rk, \textit{A Characterization
	of Chen Surfaces in $E^{4}$.} J. Amer. Math. Soc. \textbf{31(2)} (2008),
209-215.

\bibitem{KA} Y. Kemer and E. Ata, \textit{Frenet-Serret \ and Darboux Frames of A
	Image Curve on a Screw Surface.}
Sinop \"{U}niversitesi Fen Bilimleri
Dergisi \textbf{2(1)} (2017), 48-58.

\bibitem{KO} \.{I}. Ki\c{s}i and G. \"{O}zt\"{u}rk, \textit{A New Approach to
	Canal Surface with Parallel Transport Frame.} International Journal of
Geometric Methods in Modern Physics. \textbf{14(2)} (2017), 1-16.

\bibitem{KO2} \.{I}. Ki\c{s}i and G. \"{O}zt\"{u}rk, \textit{A New Type of
	Tubular Surface Having Pointwise 1-Type Gauss Map in Euclidean 4-Space $\mathbb{E}^{4}$.} Journal of the Korean Mathematica Society. \textbf{55(4)} (2018), 923-938.

\bibitem{KOA} \.{I}. Ki\c{s}i, G. \"{O}zt\"{u}rk and K. Arslan, \textit{A New
	Type of Canal Surface in Euclidean Space $E^{4}$.} arXiv:1502.06947v2, 2016.

\bibitem{KT} \.{I}. Ki\c{s}i and M. Tosun, \textit{Spinor Darboux Equations of
	Curves in Euclidean 3-Space.} Mathematica Moravica. \textbf{19(1)} (2015), 87-93.

\bibitem{MPSY} T. Maekawa, M.N. Patrikalakis, T. Sakkalis and G. Yu, \textit{Analysis and applications of pipe surfaces.}, Comput. Aided Geom. Design. \textbf{15(5)} (1998), 437-458.

\bibitem{O} B. \"{O}zdemir, \textit{A Characterization of Focal Curves and
	Focal Surfaces in $\mathbb{E}^{n}$}. PhD Thesis, Uludag University, Bursa, Turkey.

\bibitem{OA2} B. \"{O}zdemir and K. Arslan, \textit{On Generalized Focal
	Surfaces in $\mathbb{E}^{3}$}, Rev. Bull. Calcutta Math. Soc. \textbf{16(1)} (2008), 23-32.

\bibitem{OA} G. \"{O}zt\"{u}rk and K. Arslan, \textit{On Focal Curves in
	Euclidean $n$\emph{-Space }$\mathbb{R}^{n}$.} Novi Sad J. Math. \textbf{48(1)} (2016), 35-44.

\bibitem{OBBA} G. \"{O}zt\"{u}rk, B. Bulca, B.K. Bayram and K. Arslan, \textit{On Canal Surfaces in $\mathbb{E}^{3}$.} Sel\c{c}uk Journal of Applied
Mathematics. \textbf{11(2)} (2010),
103-108.

\bibitem{S} A.H. Sorour, \textit{Weingarten tube-like surfaces in Euclidean $%
	3 $\emph{-space,}.} Stud. Univ. Babe\c{s}-Bolyai Math. \textbf{61(2)} (2016),
239-250.
\end{thebibliography}
\end{document}